\documentclass[12pt]{article}
\usepackage{tikz}
\usepackage{amsmath}
\usepackage{indentfirst, latexsym, bm, amssymb}

\begin{document}

\newtheorem{theorem}{Theorem}[section]
\newtheorem{corollary}[theorem]{Corollary}
\newtheorem{definition}[theorem]{Definition}
\newtheorem{conjecture}[theorem]{Conjecture}
\newtheorem{question}[theorem]{Question}
\newtheorem{problem}[theorem]{Problem}
\newtheorem{lemma}[theorem]{Lemma}
\newtheorem{proposition}[theorem]{Proposition}
\newtheorem{example}[theorem]{Example}
\newenvironment{proof}{\noindent {\bf
Proof.}}{\rule{3mm}{3mm}\par\medskip}
\newcommand{\remark}{\medskip\par\noindent {\bf Remark.~~}}
\newcommand{\pp}{{\it p.}}
\newcommand{\de}{\em}

\newcommand{\JEC}{{\it Europ. J. Combinatorics},  }
\newcommand{\JCTB}{{\it J. Combin. Theory Ser. B.}, }
\newcommand{\JCT}{{\it J. Combin. Theory}, }
\newcommand{\JGT}{{\it J. Graph Theory}, }
\newcommand{\ComHung}{{\it Combinatorica}, }
\newcommand{\DM}{{\it Discrete Math.}, }
\newcommand{\ARS}{{\it Ars Combin.}, }
\newcommand{\SIAMDM}{{\it SIAM J. Discrete Math.}, }
\newcommand{\SIAMADM}{{\it SIAM J. Algebraic Discrete Methods}, }
\newcommand{\SIAMC}{{\it SIAM J. Comput.}, }
\newcommand{\ConAMS}{{\it Contemp. Math. AMS}, }
\newcommand{\TransAMS}{{\it Trans. Amer. Math. Soc.}, }
\newcommand{\AnDM}{{\it Ann. Discrete Math.}, }
\newcommand{\NBS}{{\it J. Res. Nat. Bur. Standards} {\rm B}, }
\newcommand{\ConNum}{{\it Congr. Numer.}, }
\newcommand{\CJM}{{\it Canad. J. Math.}, }
\newcommand{\JLMS}{{\it J. London Math. Soc.}, }
\newcommand{\PLMS}{{\it Proc. London Math. Soc.}, }
\newcommand{\PAMS}{{\it Proc. Amer. Math. Soc.}, }
\newcommand{\JCMCC}{{\it J. Combin. Math. Combin. Comput.}, }
\newcommand{\GC}{{\it Graphs Combin.}, }

\title{ Spectral Radius of Uniform Hypergraphs and Degree Sequences
\thanks{
This work is supported by the National Natural Science Foundation of China (Nos.11531001 and 11271256),  the Joint NSFC-ISF Research Program (jointly funded by the National Natural Science Foundation of China and the Israel Science Foundation (No. 11561141001)),  Innovation Program of Shanghai Municipal Education Commission (No. 14ZZ016) and Specialized Research Fund for the Doctoral Program of Higher Education (No.20130073110075).  }}
\author{  Dong-Mei Chen$^{1}$, Zhi-Bing  Chen$^{1}$, Xiao-Dong Zhang$^{2}$\thanks{Corresponding  author ({\it E-mail address:}
xiaodong@sjtu.edu.cn)}
\\
{\small $^{1}$College of Mathematics and Statistics}\\
{\small Shenzhen University} \\
{\small   3688 Nanhai Road, Shenzhen 518060,  P.R. China}\\
{\small $^{2}$School fo Mathematical Sciences,  MOE-LSC, SHL-MAC}\\
{\small Shanghai Jiao Tong University} \\
{\small  800 Dongchuan road, Shanghai, 200240,  P.R. China}\\
 }
\date{}
\maketitle
 \begin{abstract}
In this paper,  we present upper bounds for the adjacency and signless Laplacian spectral radii of uniform hypergraphs in terms of degree sequences. 
 \end{abstract}


\vskip 0.5cm

\section{Introduction}

Let ${\mathcal{H}}=(V(\mathcal{H}), E(\mathcal{H}))$ be a simple (i.e., no loops or multiedges) hypergraph, where the vertex set $V(\mathcal{H})=[n]:=\{1, \cdots, n\}$ and the edge set $E(\mathcal{H})=\{e_1, \cdots, e_m\}$ with $e_i\subseteq V(\mathcal{H})$ for $i=1, \cdots, m$.  Further, if $|e_i|=k$ for $i=1, \cdots, m$, then $\mathcal{H}$ is called a {\it $k-$uniform hypergraph}.  The {\it degree} of vertex $v\in V(\mathcal{H})$ in a hypergraph $\mathcal{H}$, written $d_v$,  is the number of edges incident to $v$, i.e., $d_v=|
\{ e\in  E(\mathcal{H}) \ \ | \ \ v\in e\}|$.  The sequence $(d_1, \cdots, d_n)$ is called {\it degree sequence} of $\mathcal{H}$.  If the degree of each vertex is equal to $d$, i.e., $d_1=\cdots=d_n=k$,  then $\mathcal{H}$ is called a {\it $d-$reguar hypergraph.}

Let $n$ and $p$ be two positive integers. An order $p$ and  dimension $n$  {\it tensor} $\mathcal{A}=(a_{i_1i_2\cdots i_p})$ over the complex field $C$ is a multidimensional array with all entries $a_{i_1i_2\cdots i_p}\in C$  for $i_j=1, \cdots, n$ and $ j=1,\cdots, p$. Clearly,
  a tensor of order $1$ and dimension $n$ is a vector and  a tensor of order $2$ and dimension $n$ is an $n\times n$ matrix. For the sake of simplicity, $a_{i_1i_2\cdots i_p}$ is denoted by $a_{i_1\alpha},$ where $\alpha=i_2\cdots i_p\in [n]^{p-1}$.
  If $\mathcal{A}$ and $\mathcal{B}$ are two tensors of order $p\ge 2$ and order $q\ge 1$, dimension $n$, respectively, the product of two tensors $\mathcal{A}$ and $\mathcal{B}$ (see \cite{shao2012}) is defined to be
the tensor $\mathcal{C}=(c_{i\alpha_1\cdots \alpha_{m-1}}) :=\mathcal{A}\cdot \mathcal{B}$ is an order $(p-1)(q-1)$ and dimension $n$, where
$$c_{i\alpha_1\cdots \alpha_{p-1}}=\sum_{i_2, \cdots, i_p=1}^na_{ii_2\cdots i_p}b_{i_2\alpha_1}\cdots b_{i_p\alpha_{p-1}}, i\in [n], \alpha_1, \cdots, \alpha_{p-1}\in [n]^{q-1}.$$

For a tensor $\mathcal{T}$ of order $p\ge 2$ and dimension $n$, if there exists a complex number  $\lambda$ and a vector $x =(x_1,\cdots, x_n)^T$ (i.e. a tensor of order $1$ and dimension $n$ such that
$$ \mathcal{T}\cdot x=\lambda x^{[p]},$$
where $x^{[p]}=(x^{p-1}, \cdots, x^{p-1})^T$,
then $\lambda$ is called an {\it eigenvalue} of $\mathcal{T}$ and $x$ is called an {\it eigenvector} of $\mathcal{T}$ corresponding to the eigenvalue $\lambda$ (for example, see \cite{qi2005}).
               The largest modulus of eigenvalues of $\mathcal{T}$ is called {\it spectral radius} of $\mathcal{T}$ and denoted by $\rho(\mathcal{T})$.   It is known  (for example, see \cite{yang2010}) that for a nonnegative tensor $\mathcal{T}$,  $\rho(\mathcal{T})$ is a nonnegative eigenvalue and corresponding to a nonnegative eigenvector. The readers may refer to an excellent survey \cite{chang2013} for spectral theory of nonnegative tensors.

               For a hypergraph $\mathcal{H}$, there are a few tensors associated with $\mathcal{H}$.  The most important tensor associated with  $\mathcal{H}$ may be the adjacency tensor.
               The {\it adjacency tensor} of a $k-$uniform hypergraph $ \mathcal{H}$ on $n$ vertices is defined as the tensor
                $\mathcal{A}(\mathcal{H})=(a_{i_1\cdots i_k})$ of order $k$ and dimension $n$,
                where
                $$a_{i_1\cdots i_n}=\left\{ \begin{array}{ll}
                \frac{1}{(k-1)!}, & \mbox{if} \ e=\{i_1, \cdots, i_k\}\in E(\mathcal{H})\\
                0, & {\mbox{otherwise}.} \end{array}\right.$$

The spectral radius of the adjacency matrix $\mathcal{A}(\mathcal{H})$ of a $k-$uniform hypergraph is called {\it spectral radius} of $\mathcal{H}$ and denoted by $\rho(\mathcal{H})$.
The spectral theory of hypergraph has received more and more attention. For example, Cooper and Dutle \cite{cooper2012} gave an excellent survey on the spectral theory of $k-$uniform hypergraph.   Another important tensor associated with $k-$uniform hypergraph $\mathcal{H}$ is
signless Laplacian tensor.
Let $ \mathcal{D}(\mathcal{H})$ be a $k-$th order $n-$dimensional diagonal tensor whose diagonal entry $d_1, \cdots, d_n$. Then $\mathcal{Q}(\mathcal{H})=\mathcal{D}(\mathcal{H})+\mathcal{A}(\mathcal{H})$ is called {\it the signless Laplacian tensor} of  $H$. The spectral radius of $\mathcal{Q}(\mathcal{H})$ is called the signless Lapalcian spectral radius of $\mathcal{H}$ and denoted by $q(\mathcal{H})$.
Li et. al. \cite{limo2016} gave some  upper bounds for the $H-$spectral radius and $Z-$spectral radius of uniform hypergraphs in terms of parameters such as number of vertices, number of edges, maximum degree, and minimum degree.
 Yuan et.al. \cite{yuanqi2016} presented upper bounds for spectral radius and signless spectral raidus of hypergraphs in terms of the degrees of vertices.  The related results may be referred to \cite{friedlane2013,  khan2015, khan2016,  lishaoqi2015, pearson2014, yuanshao2016}.

  In this paper, we present some upper bound for spectral radius and signless Laplacian spectral radius of a $k-$unform hypergraph in terms of degree sequences, which extends some known results on hypergraphs. In Section 2, Some preliminaries and Lemmas are presented. In Section 3,  we present the main results of this paper and proof.

\section{Preliminaries}
In this section, we present some known results and lemmas
\begin{lemma}\label{max-min} \cite{yang2010, khan2015}
 Let $\mathcal{T}=(a_{i_1\cdots i_k})$ be an order $k$ and dimension $n$ tensor with  $k\ge 2$.  Then
  \begin{equation}\label{max-e1}
  \min\{ R_i(\mathcal{T}): 1\le i\le n\} \le \rho (\mathcal{T})\le \max\{ R_i(\mathcal{T}) : 1\le i\le n\},
  \end{equation}
  where $R_i(\mathcal{T})$ is the sum of row $i$ in $\mathcal{T}$, i.e., $R_i(\mathcal{T})=\sum_{i_2,\cdots, i_k=1}^{n}a_{ii_2\cdots i_k}$.
  Moreover, if $\mathcal{T}$ is weakly irreducible,  then either equality holds if and only if $R_1(\mathcal{T})=\cdots=R_n(\mathcal{T})$.
    \end{lemma}

  \begin{lemma}\label{diag-sim}(\cite{shao2012})
  Let $\mathcal{A}$ and $\mathcal{B}$ be two order $k$ dimension $n$ tensors. If there exists a nonsingular diagonal matrix $D$ such that
  $\mathcal{B}=D^{-(k-1)} \cdot\mathcal{A} \cdot D$, then $\mathcal{A}$ and $\mathcal{B}$ have the same eigenvalues including multiplicity. In particular, they have the same spectral radius.
\end{lemma}


For convenience, if $n$ and $k$ are two integers,
$$ \binom{n}{k}=\left\{ \begin{array}{ll}
\frac{n!}{k!(n-k)!},& \mbox {if $n\ge k\ge 0$}\\
0 & \mbox{ if $n<k$ or $k<0$.}
\end{array}\right.$$

 The following identity equality is known.
\begin{lemma} (\cite{lovasz2003})\label{id} Let $n> s\ge 3$ and $kge 2$ be three positive integers. Then
\begin{equation}\label{A1=A3}
\sum_{r=0}^{k-2}\binom{s-3}{k-r-2}\binom{n-s+1}{r}=\sum_{p=1}^{k-1}\binom{s-2}{p-1}\binom{n-p}{k-p-1}=\binom{n-2}{k-2}.
\end{equation}

\end{lemma}

\begin{lemma} If $n, k, s$ are three positive integers with  $n\ge s\ge 3$, then
\begin{equation}\label{A4=A2+}
\sum_{r=0}^{k-2}(r+1)\binom{s-1}{r+1}\binom{n-s}{k-r-2}=\sum_{r=0}^{k-2}(k-1-r)\binom{s-2}{k-r-1}\binom{n-s+1}{r}+\binom{n-2}{k-2}.
\end{equation}
\end{lemma}
\begin{proof}
Since $(1+x)^{n-1}=(1+x)^{s-1}(1+x)^{n-s}, $ we differentiate both side with respect to $x$ and have
\begin{equation}\label{A3A4-e2}
\sum_{r=1}^{n-1}r\binom{n-1}{r}x^{r-1}=\sum_{r=1}^{s-1}r\binom{s-1}{r}x^{r-1}\sum_{r=0}^{n-s}\binom{n-s}{r}x^{r}
+\sum_{r=0}^{s-1}\binom{s-1}{r}x^{r}\sum_{r=1}^{n-s}r\binom{n-s}{r}x^{r-1}.
\end{equation}
Hence considering the coefficients of $x^{k-2}$ of both side of equation (\ref{A3A4-e2}), we have
\begin{equation}\label{A3A4-e3}(k-1)\binom{n-1}{k-1}=\sum_{r=0}^{k-2}(r+1)\binom{s-1}{r+1}\binom{n-s}{k-2-r}+\sum_{r=0}^{k-2}(k-1-r)\binom{s-1}{r}
\binom{n-s}{k-1-r}.\end{equation}
On the other hand,
Since $(1+x)^{n-1}=(1+x)^{s-2}(1+x)(1+x)^{n-s}, $  we differentiate both side with respect to $x$ and have
\begin{equation}\label{A3A4-e4}
\left[(1+1)^{n-1}\right]^{\prime}= \left[(1+x)^{s-2}\right]^{\prime}(1+x)^{n-s+1}+(1+x)^{s-1}\left[(1+x)^{n-s}\right]^{\prime}+(1+x)(1+x)^{n-2}.
\end{equation}
Then considering the coefficients of $x^{k-2}$ both side of equation (\ref{A3A4-e4}), we have
\begin{eqnarray}\label{A3A4-e5}
(k-1)\binom{n-1}{k-1}&=&\sum_{r=0}^{k-2}(k-r-1)\binom{s-2}{k-r-1}\binom{n-s+1}{r}\nonumber\\
&&+\sum_{r=0}^{k-2}(k-r-1)\binom{s-1}{r}\binom{n-s}{k-r-1}+\binom{n-2}{k-2}.
\end{eqnarray}
By (\ref{A3A4-e3}) and (\ref{A3A4-e5}), it is easy to see that (\ref{A3A4-e2}) holds.
\end{proof}
\section{Main Results}
In this section, we present the main result of this paper as follows.
\begin{theorem}\label{main}
Let $\mathcal{H}$ be a $k-$uniform hypergraph with degree sequence $d_1\ge d_2\ge\cdots\ge d_n$.  Denote $A_1=\frac{1}{k-1}\binom{n-2}{k-2}$, $A_2=\sum_{r=0}^{k-2}\frac{r+1}{k-1}\binom{s-1}{r+1}\binom{n-s}{k-r-2}$, $\Delta=(d_s+sA_1-A_2)^2+4A_1\sum_{t=1}^{s-1}(d_t-d_s)$ and $\phi_s=\frac{d_s-A_2+(s-2)A_1+\sqrt{\Delta}}{2} $, for $1\le s\le n$.
Then
\begin{equation}\label{upper-e1}
\rho(\mathcal{H})\le \min_{1\le s\le n}\left\{ \phi_s=\frac{d_s-A_2+(s-2)A_1+\sqrt{\Delta}}{2}\right
\}.\end{equation}

\end{theorem}
\begin{proof}  We first consider $s\ge 3$.
Let  $\mathcal{A}$ be the adjacency matrix of $\mathcal{H}$ and $D=diag(x_1, \cdots, x_n)$ be the diagonal matrix with $x_s=\cdots =x_n=1$.
Let $\mathcal{B}=D^{-(k-1)}\cdot\mathcal{A}\cdot D=(b_{i_1\cdots i_k})$. By Lemma~\ref{diag-sim},  $\rho(\mathcal{H})=\rho(\mathcal{A})=\rho(\mathcal{B})$.  For $0\le r\le n$, denote
$I_r=\{ \{ i_2, \cdots, i_k \} \ | \  1\le i_2, \cdots, i_k\le n, \mbox{ there are  exactly }\ r\ \mbox{elements in }\  \{i_2, \cdots, i_k\}$  with at least $\  s \}$.  The sum of row $j$ in the tensor $\mathcal{B}$ is denoted by $R_j(\mathcal{B})$ for $1\le j\le n$.
  We consider the following two cases

 {\bf Case 1:} $1\le j\le s-1$.  By (\ref{A1=A3}) in Lemma~\ref{id}, $x_s=\cdots=x_n=1$, we have

\begin{eqnarray*}
 R_j(\mathcal{B})&=& \sum_{i_2,\cdots, i_k=1}^{n}b_{ji_2\cdots i_k}\\
 &=&\frac{1}{x_j^{k-1}} \sum_{i_2,\cdots, i_k=1}^{n}a_{ji_2\cdots i_k}x_{i_2}\cdots x_{i_k}\\
 &=& \frac{1}{x_j^{k-1}} \sum_{i_2,\cdots, i_k=1}^{n}a_{ji_2\cdots i_k}x_{i_2}\cdots x_{i_k}\\
 &=&  \frac{1}{x_j^{k-1}}\left\{ \sum_{r=0}^{k-2}\sum_{\{j,i_2,\cdots, i_k\}\in E(\mathcal{H}), \{j_2, \cdots, j_k\}\in I_r}x_{i_2}\cdots x_{i_k}+
 \sum_{\{j, i_2, \cdots, j_k\}\in E(\mathcal{H}), \{j_2, \cdots, j_k\}\in I_{k-1}}1\right\}\\
  &=& \frac{1}{x_j^{k-1}}\left\{ \sum_{r=0}^{k-2}\sum_{\{j,i_2,\cdots, i_k\}\in E(\mathcal{H}), \{j_2, \cdots, j_k\}\in I_r}(x_{i_2}\cdots x_{i_k}-1)+
 d_j\right\}\\
  \end{eqnarray*}
 \begin{eqnarray*}
  &\le& \frac{1}{x_j^{k-1}}\left\{ \sum_{r=0}^{k-2}\sum_{\{j,i_2,\cdots, i_k\}\in E(\mathcal{H}), \{j_2, \cdots, j_k\}\in I_r}\left(\frac{x_{i_2}^{k-1}+\cdots+x_{i_k}^{k-1}}{k-1}-1\right)+
 d_j\right\}\\
 &\le& \frac{1}{x_j^{k-1}}\left\{ \sum_{r=0}^{k-2}\sum_{ \{j_2, \cdots, j_k\}\in I_r}\left(\frac{x_{i_2}^{k-1}+\cdots+x_{i_k}^{k-1}}{k-1}-1\right)+
 d_j\right\}\\
 &=& \frac{1}{x_j^{k-1}}\left\{ d_j+\frac{1}{k-1}\sum_{r=0}^{k-2}\binom{s-3}{k-r-2}\binom{n-s+1}{r}\left[\sum_{t=1}^{s-1}x_t^{k-1}-x_j^{k-1}\right]\right.\\
 &&+\frac{1}{k-1}\sum_{r=0}^{k-2}\binom{s-2}{k-r-1}\binom{n-s}{r-1}\sum_{t=s}^nx_t^{k-1}-\left.\sum_{r=0}^{k-2}\binom{s-2}{k-r-1}\binom{n-s+1}{r}\right\} \\
 &=& \frac{1}{x_j^{k-1}}\left\{ d_j+\frac{1}{k-1}\binom{n-2}{k-2}\left[\sum_{t=1}^{s-1}x_t^{k-1}-x_j^{k-1}\right]\right.\\
 &&-\left.\frac{1}{k-1}\sum_{r=0}^{k-2}\frac{k-1-r}{k-1}\binom{s-2}{k-r-1}\binom{n-s+1}{r-1}\right\}. \\
  &=& \frac{1}{x_j^{k-1}}\left\{ d_j+A_1\left[\sum_{t=1}^{s-1}x_t^{k-1}-x_j^{k-1}\right]-A_2+A_1\right\}. \\
 \end{eqnarray*}
 {\bf Case 2:} $s\le j\le n$.  By (\ref{A1=A3}) in Lemma~\ref{id}, $x_s=\cdots=x_n=1$, we have

 \begin{eqnarray*}
 R_j(\mathcal{B}) &=& \sum_{i_2,\cdots, i_k=1}^{n}b_{ji_2\cdots i_k}\\
 &=& \sum_{i_2,\cdots, i_k=1}^{n}a_{ji_2\cdots i_k}x_{i_2}\cdots x_{i_k}\\
 &=& \sum_{i_2,\cdots, i_k=1}^{n}a_{ji_2\cdots i_k}x_{i_2}\cdots x_{i_k}\\
 &=&   \sum_{r=0}^{k-2}\sum_{\{j,i_2,\cdots, i_k\}\in E(\mathcal{H}), \{j_2, \cdots, j_k\}\in I_r}x_{i_2}\cdots x_{i_k}+
 \sum_{\{j, i_2, \cdots, j_k\}\in E(\mathcal{H}), \{j_2, \cdots, j_k\}\in I_{k-1}}1\\
  &=& d_j+ \sum_{r=0}^{k-2}\sum_{\{j,i_2,\cdots, i_k\}\in E(\mathcal{H}), \{j_2, \cdots, j_k\}\in I_r}(x_{i_2}\cdots x_{i_k}-1)
 \\
 \end{eqnarray*}
 \begin{eqnarray*}
  &\le& d_s+ \sum_{r=0}^{k-2}\sum_{\{j,i_2,\cdots, i_k\}\in E(\mathcal{H}), \{j_2, \cdots, j_k\}\in I_r}\left(\frac{x_{i_2}^{k-1}+\cdots+x_{i_k}^{k-1}}{k-1}-1\right)\\
  &=&  d_s+\frac{1}{k-1}\sum_{r=0}^{k-2}\binom{s-2}{r}\binom{n-s}{k-r-2}\sum_{t=1}^{s-1}x_t^{k-1}\\
 &&+\frac{1}{k-1}\sum_{r=0}^{k-2}\binom{s-1}{r+1}\binom{n-s-1}{k-r-3}\left[\sum_{t=s}^nx_t^{k-1}-x_j^{k-1}\right]-
 \sum_{r=0}^{k-2}\binom{s-1}{r+1}\binom{n-s}{k-r-2} \\
 &=&  d_s+\frac{1}{k-1}\binom{n-2}{k-2}\sum_{t=1}^{s-1}x_t^{k-1}-\sum_{r=0}^{k-2}\frac{r+1}{k-1}\binom{s-1}{r+1}\binom{n-s}{k-r-2}. \\
 &=&  d_s+A_1\sum_{t=1}^{s-1}x_t^{k-1}-A_2. \\
      \end{eqnarray*}

     Let $x_t^{k-1}=1+\frac{d_t-d_s}{A_1+\phi_s}$. By the definition of $\phi_s$,  it is easy to see that
     $$\phi_s^2-[d_s-A_2+(s-2)A_1]\phi_s-A_1[(d_s-A_2+(s-1)A_1+\sum_{t=1}^{s-1}(d_t-d_s)]=0.$$
     Hence  $d_s+A_1\sum_{t=1}^{s-1}x_t^{k-1}-A_2=\phi_s$, which implies that $R_j(\mathcal{B})\le \phi_s$ for $s\le j\le n$.
     Moreover, for  $1\le j\le s-1$,  we have
     $$\sum_{t=1}^{s-1}x_t^{k-1}=\frac{\phi_s-d_s+A_2}{A_1}$$
     and
     \begin{eqnarray*}
R_j(\mathcal{B})&\le & \frac{1}{x_j^{k-1}}\left\{ d_j+A_1\left[\sum_{t=1}^{s-1}x_t^{k-1}-x_j^{k-1}\right]-A_2+A_1\right\}\\
&=& \frac{1}{1+\frac{d_j-d_s}{A_1+\phi_s}}\left\{ d_j+\phi_s-d_s+A_2-A_1\left(1+\frac{d_j-d_s}{A_1+\phi_s}\right)-A_2+A_1\right\}\\
&=&\phi_s.
\end{eqnarray*}
Hence by Lemma~\ref{max-min}, we have $\rho(\mathcal{B})\le \phi_s$ for $3\le s\le n$.

  If $s=1$, then it is easy to see that $\phi_1=d_1$. By \cite{cooper2012},  we have $\rho(\mathcal{H})\le d_1=\phi_1$.

  If $s=2$, then  by the same argument as $s\ge 3$, it is easy to see that
  \begin{equation}\label{s2-e1}
  R_1(\mathcal{B})\le \frac{d_1}{x_1^{k-1}}
  \end{equation}
  and
  \begin{equation}\label{s2-e2}
  R_j(\mathcal{H})\le d_2+A_1x_1^{k-1}-A_1, \mbox{for}\ j=2, \cdots, n.
  \end{equation}
  Let $x_1^{k-1}=\frac{-(d_2-A_1)+\sqrt{(d_2-A_1)^2+4A_1d_1}}{2A_1}$. Then
  $$R_1(\mathcal{B})\le \frac{d_1}{x_1^{k-1}}=\phi_2$$
  and $$R_j(\mathcal{H})\le d_2+A_1x_1^{k-1}-A_1=\phi_2,\  \mbox{for} \ j=2, \cdots, n.$$
Hence $\rho(\mathcal{H})\le \phi_2$.
 Therefore,
$\rho(\mathcal{H})\le \phi_s$ for $s=1, \cdots, n$.
\end{proof}
{\bf Remark}  The The sequence $\phi_1, \cdots, \phi_n$ is not necessarily non-increasing.   In particular, we are able to get an upper bound
in terms of  then minimum degree $\delta$ and the size of edges.
\begin{corollary}
Let $\mathcal{H}$ be a $k-$uniform hypergraph with the minimum degree $\delta$ and the size $m$ of edges. Then
$$\rho(\mathcal{H})\le \frac{\delta+\frac{n-2}{k-1}\binom{n-2}{k-2}-\binom{n-1}{k-2}+\sqrt{\left(\delta+\frac{n}{k-1}\binom{n-2}{k-2}-\binom{n-1}{k-2}\right)^2+\frac{4}{k-1}
\binom{n-2}{k-2}(km-n\delta)}}{2}.$$
 \end{corollary}
\begin{proof}
The assertion follows from $\rho(H)\le \phi_n$ in Theorem~\ref{main}.
\end{proof}

 Similarly, we are able to get an upper bound for the signless Laplacian spectral radius of $Q(\mathcal{H})$.
 \begin{theorem}\label{signless}
Let $\mathcal{H}$ be a $k-$uniform hypergraph with degree sequence $d_1\ge d_2\ge\cdots\ge d_n$ and
 $q(\mathcal{H})$ be the spectral radius of $Q(\mathcal{H})$. Denote
 $A_1=\frac{1}{k-1}\binom{n-2}{k-2}$, $A_2=\sum_{r=0}^{k-2}\frac{r+1}{k-1}\binom{s-1}{r+1}\binom{n-s}{k-r-2}$, $\Theta=(2d_s+sA_1-A_2)^2+8A_1\sum_{t=1}^{s-1}(d_t-d_s)$ and $\psi_s=\frac{2d_s-A_2+(s-2)A_1+\sqrt{\Theta}}{2} $, for $1\le s\le n$.
Then
\begin{equation}\label{upper-e1}
q(\mathcal{H})\le \min_{1\le s\le n}\left\{ \psi_s=\frac{2d_s-A_2+(s-2)A_1+\sqrt{\Theta}}{2}\right
\}.\end{equation}
\end{theorem}
 \begin{proof} 
 Clearly, $\psi_1=2d_1$. Hence $q(\mathcal{H})\le \psi_1$ by \cite{yuan2015}.
 Let $Y=diag(y_1, 1, \cdots, 1)$ be the diagonal matrix  with $y_1^{k-1}=
 \frac{-(2d_2-A_1)+\sqrt{(2d_2-A_1)^2+8A_1d_1}}{2A_1}$
  and $\mathcal{M}=Y^{k-1}\cdot q(\mathcal{H}) \cdot Y$.
 Then by the same argument in Theorem~\ref{main}, it is easy to see that
  $$R_j(\mathcal{M})\le \psi_2 \ \mbox{for}\ j=1, \cdots, n.$$
 Hence $q(\mathcal{H})\le \psi_2.$
 
  We now consider $s\ge 3$.
 Let $X=diag(x_1, \cdots, x_n)$  be the diagonal matrix with $x_j=1+\frac{2(d_j-d_s)}{A_1+\psi_s}$ for $j=1, \cdots, s-1$ and $x_s=\cdots= x_n=1$.
Let $\mathcal{C}=X^{-(k-1)}\cdot Q(\mathcal{H}) \cdot X$.  For $ 1\le j\le s-1$, by the same argument in Theorem~\ref{main}, we have
$$R_j(\mathcal{C})\le  \psi_s \ \mbox{for}\ j=1, \cdots, n.$$
Hence $q(\mathcal{H})\le\psi_s$.
So the assertion holds.
\end{proof}
\begin{corollary}
Let $\mathcal{H}$ be a $k-$uniform hypergraph with the minimum degree $\delta$ and the size $m$ of edges. Then
$$q(\mathcal{H})\le \frac{2\delta+\frac{n-2}{k-1}\binom{n-2}{k-2}-\binom{n-1}{k-2}+\sqrt{\left(2\delta+\frac{n}{k-1}\binom{n-2}{k-2}-\binom{n-1}{k-2}\right)^2+\frac{8}{k-1}
\binom{n-2}{k-2}(km-n\delta)}}{2}.$$
 \end{corollary}

\end{document}